\documentclass[a4paper]{article}
\usepackage{latexsym}
\usepackage{indentfirst}
\usepackage{graphicx}
\usepackage{amsmath}
\usepackage{amssymb}
\usepackage{url}
\usepackage{arxiv}
\usepackage{csquotes}
\usepackage{amsthm}

\newtheorem{defn}{Definition}[section]
\newtheorem{exam}{Example}[section]
\newtheorem{lem}{Lemma}[section]

\newtheorem{thm}{Theorem}[section]
\newtheorem{prop}{Proposition}[section]
\newtheorem{cor}{Corollary}[section]

\newcommand{\RomanNumeralCaps}[1]
    {\MakeUppercase{\romannumeral #1}}

\begin{document}
\title{Edwards curve points counting method and supersingular Edwards and Montgomery curves  \RomanNumeralCaps{3}\\ }

\author{Ruslan Skuratovskii$^1$  \\ 
\texttt{r.skuratovskii@kpi.ua,  ruslan@ubicyb.kiev.ua} \\ 
National Technical University of Ukraine "KPI im. Igor Sikorsky ", Institute of Cybernetics of NASU.
}

\date{\today}
\maketitle

\begin{abstract}
We consider algebraic affine and projective curves of Edwards \cite{E, SkOdProj} over the finite field ${{\text{F}}_{{{p}^{n}}}}$. It is known that many modern cryptosystems \cite{SkBlock} can be naturally transformed into elliptic curves \cite{Kob}. We research Edwards algebraic curves over a finite field, which are one of the most promising supports of sets of points which are used for fast group operations \cite{Bir}. We construct a new method for counting the order of an Edwards curve over a finite field. It should be noted that this method can be applied to the order of elliptic curves due to the birational equivalence between elliptic curves and Edwards curves.

We not only find a specific set of coefficients with corresponding field characteristics for which these curves are supersingular, but we additionally find a general formula by which one can determine whether a curve ${{E}_{d}}[{{\mathbb{F}}_{p}}]$ is supersingular over this field or not.

The embedding degree of the supersingular curve of Edwards over ${{\mathbb{F}}_{{{p}^{n}}}}$ in a finite field is investigated and the field characteristic, where this degree is minimal, is found.
A birational isomorphism between the Montgomery curve and the Edwards curve is also constructed.
A one-to-one correspondence between the Edwards supersingular curves and Montgomery supersingular curves is established.

The criterion of supersingularity for Edwards curves is found over ${{\mathbb{F}}_{{{p}^{n}}}}$. 
\end{abstract}

\section{Introduction}
The task of finding the order of an algebraic curve over a finite field ${{\mathbb{F}}_{{{p}^{n}}}}$ is now very relevant and is at the center of many mathematical studies in connection with the use of groups of points of curves of genus 1 in cryptography.
In our article, this problem is solved for the Edwards and Montgomery curves.

 The criterion of supersingularity of the Edwards curves is found over ${{\mathbb{F}}_{{{p}^{n}}}}$. We additionally propose a method for counting the points from Edwards curves and elliptic curves in response to an earlier paper by Schoof \cite{Schoof}.

We consider the algebraic affine and projective Edwards curves over a finite field. We not only find a specific set of coefficients with corresponding field characteristics for which these curves are supersingular, but we additionally find a general formula by which one can determine whether a curve ${{E}_{d}}[{{\mathbb{F}}_{p}}]$ is supersingular over this field or not.

\section{Main Result}
The twisted Edwards curve with coefficients $a,d \in {{F}_{p}}^{*}$ is the curve $E_{a,d}:$
$$a{{x}^{2}}+{{y}^{2}}=1+d{{x}^{2}}{{y}^{2}},$$ where $ad(a-d)\ne 0$, $d \ne 1$, $p\ne 2$ and $a\ne d$.
It should be noted that a twisted Edwards curve is called an Edwards curve when $a=1$.
We denote by ${{E}_{d}}$ the Edwards curve with coefficient $d\in {{F}_{p}}^{*}$ which is defined as
$${{x}^{2}}+{{y}^{2}}=1+d{{x}^{2}}{{y}^{2}},$$ over ${{\mathbb{F}}_{p}}$. The projective curve has form $F(x,y,z)=a{{x}^{2}}{{z}^{2}}+{{y}^{2}}{{z}^{2}}={{z}^{4}}+d{{x}^{2}}{{y}^{2}}$.
The special points are the infinitely distant points $(1,0,0)$ and $(0,1,0)$ and therefore we find its singularities at infinity in the corresponding affine components
${{A}^{1}}:= a{{z}^{2}}+{{y}^{2}}{{z}^{2}}={{z}^{4}}+d{{y}^{2}}$  and  ${{A}^{2}}:=\,\,\,a{{x}^{2}}{{z}^{2}}+{{z}^{2}}={{z}^{4}}+d{{x}^{2}}$. These are simple singularities.

We describe the structure of the local ring at the point ${{p}_{1}}$ whose elements are quotients of functions with the form $F(x,y,z)=\frac{f(x,y,z)}{g(x,y,z)}$, where the denominator cannot take the value of $0$ at the singular point ${{p}_{1}}$. In particular, we note that a local ring which has two singularities consists of functions with the denominators are not divisible by $(x-1)(y-1)$.

We denote by ${{\delta }_{\text{p}}}=\dim{}^{{{{\overline{\mathcal{O}}}}_{\text{p}}}}/{}_{{{\mathcal{O}}_{p}}}$, where ${{\mathcal{O}}_{\text{p}}}$ denotes the local ring at the singular point $p$ which is generated by the relations of regular functions ${{\mathcal{O}}_{\text{p}}}=\left\{ \frac{f}{g}:\,\,(g,(x-1)(y-1))=1 \right\}$ and ${{\overline{\mathcal{O}}}_{\text{p}}}$ denotes the whole closure of the local ring at the singular point $p$.

We find that ${{\delta }_{p}}=\dim{}^{{{{\overline{\mathcal{O}}}}_{p}}}/{}_{{{\mathcal{O}}_{p}}}=1$ is the dimension of the factor as a vector space. Because the basis of extension ${{{{\overline{\mathcal{O}}}}_{p}}}/{}_{{{\mathcal{O}}_{p}}}$ consists of just one element at each distinct point, we obtain that ${{\delta }_{p}}=1$. We then calculate the genus of the curve according to Fulton \cite{F}
$${{\rho }^{*}}(C)={{\rho }_{\alpha }}(C)-\sum\limits_{p\in E}{{{\delta }_{p}}}=\,\,\frac{(n-1)(n-2)}{2}-\sum\limits_{p\in E}{{{\delta }_{p}}}=3-2=1,$$
where ${{\rho }_{\alpha }}(C)$ denotes the arithmetic genus of the curve $C$ with parameter $n=\text{deg}(C)=4$. It should be noted that the supersingular points were discovered in \cite{SkRMM}. Recall the curve has a genus of 1 and as such it is known to be isomorphic to a flat cubic curve, however, the curve is importantly not elliptic because of its singularity in the projective part.

Both the Edwards curve and the twisted Edwards curve are isomorphic to some affine part of the elliptic curve. The Edwards curve after normalization is precisely a curve in the Weierstrass normal form, which was proposed by Montgomery \cite{Bir} and will be denoted by ${{E}_{M}}$.

 Koblitz \cite{Kob, F} tells us that one can detect if a curve is supersingular using the search for the curve when that curve has the same number of points as its torsion curve. Also an elliptic curve $E$ over $F_q$ is called supersingular if for every finite extension $F_{q^r}$ there are no points in the group
$E(F_{q^r} )$ of order $p$ \cite{Silv}. It is known \cite{Bir} that the transition from an Edwards curve to the related torsion curve is determined by the reflection $\left( \overline{x},\overline{y} \right)\mapsto \left( x,y \right)=\left( \overline{x},\frac{1}{{\overline{y}}} \right)$.

We now recall an important result from Vinogradov \cite{Vinogradov} which will act as criterion for supersingularity.

\begin{lem}[Vinogradov \cite{Vinogradov}] \label{Sum}
 Let $k\in \mathbb{N}$ and $p \in \mathbb{P}$.  Then
$$\sum\limits_{k=1}^{p-1}{{{k}^{n}}}\equiv \ \left\{ \begin{array}{*{35}{l}}
   0 \text{  } (\bmod \text{ } p), & n\not|(p-1),  \\
   -1 \text{ } (\bmod \text{ } p), & n | (p-1), \text{ }  \\
\end{array} \right.$$
\end{lem}
where  $n | (p-1)$ denotes that $n$ is divisible by $p-1$.
\vspace{5.0mm}

The \textit{order of a curve} is precisely the number of its affine points with a neutral element, where the group operation is well defined. It is known that the order of ${{x}^{2}}+{{y}^{2}}=1+d{{x}^{2}}{{y}^{2}}$ coincides with the order of the curve ${{x}^{2}}+{{y}^{2}}=1+{{d}^{-1}}{{x}^{2}}{{y}^{2}}$ over finite field ${{F}_{p}}$.

We will now strengthen an existing result given in \cite{SkRMM}. We denote the \textit{number of points with a neutral element of an affine Edwards} curve over the finite field ${{\mathbb{F}}_{p}}$ by $N_{d[p]}$ and the \textit{number of points on the projective curve} over the same field by $\overline{N}_{d[p]}$.

\begin{thm} \label{first result}
If $p \equiv 3 (\bmod \text{ } 4)$ is prime and
the following condition of supersingularity
\begin{equation} \label{Super}
\begin{array}{l}
\sum\limits_{j=0}^{\frac{p-1}{2}}{{{(C_{\frac{p-1}{2}}^{j})}^{2}}{{d}^{j}}}\equiv 0(\bmod \text{ } p),
\end{array}
\end{equation}
is true
then the orders of the curves ${{x}^{2}}+{{y}^{2}}=1+d{{x}^{2}}{{y}^{2}}$ and ${{x}^{2}}+{{y}^{2}}=1+{{d}^{-1}}{{x}^{2}}{{y}^{2}}$ over ${{F}_{p}}$ are equal to
$${{N}_{{{{d [p]}}}}}=p+1,$$
 when $\left({\frac{d}{p}}\right)=-1$, and
$${{N}_{d[p]}} = p-3,$$
 when $\left({\frac{d}{p}}\right)=1$.

\end{thm}


\begin{proof}
Consider the curve $E_d$:
\begin{equation} \label{Eqw}
{{x}^{2}}+{{y}^{2}}=1+d{{x}^{2}}{{y}^{2}}.
\end{equation}

Transform it into the form ${{y}^{2}}(1 - d{{x}^{2}}{{y}^{2}}) = 1 - {{x}^{2}}$, then we express $y^2$ by applying a rational transformation which lead us to the curve ${{y}^{2}} =\frac{ 1 - {{x}^{2}}} {1 - d{{x}^{2}}{{y}^{2}}}$.
For analysis we transform it into the curve
\begin{equation} \label{Prod}
 {{y}^{2}}=({{x}^{2}}-1)(d{{x}^{2}}-1).
\end{equation}

We denote the number of points from an affine Edwards curve over the finite field ${{\mathbb{F}}_{p}}$ by $M_{d[p]}$.
This curve (\ref{Prod}) has $M_{d[p]} = N_{d[p]}+ \left({\frac{d}{p}}\right) +1$ points, which is precisely $\left({\frac{d}{p}}\right) +1 $ greater than the number of points of curve $E_d$. Note that $\left({\frac{d}{p}}\right)$ denotes the Legendre Symbol.
Let $a_0, a_1, \dots, a_{2p - 2}$ be the coefficients of the polynomial $a_0 + a_1 x + \dots + a_{2p - 2} x^{2p - 2}$, which was obtained from $({{{{x} ^ {2}} - 1}}) ^ {\frac{p -1} {2}} {{(d {{x} ^ {2}} - 1)} ^ {\frac {p-1} {2}}}$ after opening the brackets.


Thus, summing over all $x$ yields $${{M}_{{{{d [p]}}}}} = \sum\limits_{x=0}^{p-1}{1+{{(({{x}^{2}}-1)(d{{x}^{2}}-1))}^{\frac{p-1}{2}}} = }p+\sum\limits_{x=0}^{p-1}{{{({{x}^{2}}-1)}^{\frac{p-1}{2}}}{{(d{{x}^{2}}-1)}^{\frac{p-1}{2}}}} \\
\equiv \sum\nolimits_{x=0}^{p-1}{{{({{x}^{2}}-1)}^{\frac{p-1}{2}}}{{(d{{x}^{2}}-1)}^{\frac{p-1}{2}}}}(\bmod \text{ } p).
$$
By opening the brackets in ${{({{x}^{2}}-1)}^{\frac{p-1}{2}}}{{(d{{x}^{2}}-1)}^{\frac{p-1}{2}}},$
we have ${{a}_{2p-2}}={{(-  1)}^{\frac{p-1}{2}}}\cdot {{d}^{\frac{p-1}{2}}}\equiv \left(\frac{d}{p}\right)(\bmod \text{ }p)$.
So, using Lemma \ref{Sum} we have
\begin{equation} \label{N_2}
{{M}_{{{{d [p]}}}}}\equiv - \left(\frac{d}{p}\right)-{{a}_{p-1}}(\bmod \text{ } p).				
\end{equation}

We need to prove that ${{M}_{{{{d [p]}}}}} \equiv 1(\bmod \text{ } p)$ if $p\equiv 3 (\bmod \text{ } 8)$ and ${{M}_{{{{d [p]}}}}} \equiv -1 (\bmod \text{ } p)$ if $p \equiv 7 (\bmod \text{ } 8)$. We therefore have to show that ${{M}_{{{{d [p]}}}}}\equiv -(\frac{d}{p})-{{a}_{p-1}}(\bmod \text{ } p)$  for $p \equiv 3 (\bmod \text{ } 4)$ if $\sum\limits_{j=0}^{\frac{p-1}{2}}{{{(C_{\frac{p-1}{2}}^{j})}^{2}}{{d}^{j}}}\equiv 0(\bmod \text{ } p)$.
If we prove that ${{a}_{p-1}}\equiv 0(\bmod \text{ } p)$, then it will follow from (\ref{Prod}). Let us determine ${{a}_{p-1}}$ according to Newton's binomial formula: $a_{p-1}$ is equal to the coefficient at $x^{p-1}$ in the polynomial, which is obtained as a product ${{({{x}^{2}}-1)}^{\frac{p-1}{2}}}{{(d{{x}^{2}}-1)}^{\frac{p-1}{2}}}$. So, $ a_{p-1} = {{(-1)}^{\frac{p-1}{2}}}\sum\limits_{j=0}^{\frac{p-1}{2}}{{{d}^{j}}{{(C_{\frac{p-1}{2}}^{j})}^2}}$. Actually, the following equality holds:

$
   \sum\limits_{j=0}^{\frac{p-1}{2}}{{{d}^{j}}(C_{\frac{p-1}{2}}^{\frac{p-1}{2}-j}){{(-1)}^{\frac{p-1}{2}-(\frac{p-1}{2}-j)}}\cdot {{d}^{j}}{{(C_{\frac{p-1}{2}}^{j})}^{2}}{{(-1)}^{\frac{p-1}{2}-j}}}= \\
    ={{(-1)}^{\frac{p-1}{2}}}
   \sum\limits_{j=0}^{\frac{p-1}{2}}{{{d}^{j}}C_{\frac{p-1}{2}}^{\frac{p-1}{2}-j}\cdot C_{\frac{p-1}{2}}^{j}=}   {{(-1)}^{\frac{p-1}{2}}}\sum\limits_{j=0}^{\frac{p-1}{2}}{{{d}^{j}}{{(C_{\frac{p-1}{2}}^{j})}^{2}}}. \\
$

Since ${{a}_{p-1}}=-\sum\limits_{j=0}^{\frac{p-1}{2}}{{{(C_{\frac{p-1}{2}}^{j})}^{2}}{{d}^{j}}}$, then exact number of affine points on non supersingular curve is the following

\begin{equation} \label{general numb}
{{M}_{{{{d [p]}}}}} \equiv  -{{a}_{2p-2}}-{{a}_{p-1}}\equiv -\left(\frac{d}{p}\right)+\sum\limits_{j=0}^{\frac{p-1}{2}}{{{(C_{\frac{p-1}{2}}^{j})}^{2}}{{d}^{j}}}(\bmod \text{ } p).
\end{equation}

 According to the condition of this theorem ${{a}_{p-1}}=0$, therefore ${{M}_{{{{d [p]}}}}} \equiv  -{{a}_{2p-2}}(\bmod p)$. Consequently, in the case when $p \equiv 3 (\bmod \text{ } 4)$, where $p$ is prime and $\sum\limits_{j=0}^{\frac{p-1}{2}}{{{(C_{\frac{p-1}{2}}^{j})}^{2}}{{d}^{j}}}\equiv 0(\bmod \text{ } p)$, the curve $E_d$ has ${{N}_{{{{d [p]}}}}} = p - \left(\frac{d}{p}\right)- \big( \left(\frac{d}{p}\right)+1\big)= p-1 - 2\left(\frac{d}{p}\right)$ affine points and a group of points of the curve completed by singular points has $p+1$ points.

Exact number of the points has upper bound $2p+1$ because for every $x \neq 0$ corresponds two valuations of $y$, but for $x=0$ we have only one solution $y=0$. Taking into account that $x\in F_p$ we have exactly $p$ values of $x$. Also there are 4 pairs $(\pm1,0)$ and $(0, \pm 1)$ which are points of $E_d$ thus ${{N}_{d[p]}}>1$. Thus ${{N}_{d[p]}}=p+1$. This completes the proof.

\begin{cor} \label{cor result}
  The orders of the curves ${{x}^{2}}+{{y}^{2}}=1+d{{x}^{2}}{{y}^{2}}$ and ${{x}^{2}}+{{y}^{2}}=1+{{d}^{-1}}{{x}^{2}}{{y}^{2}}$ over ${{F}_{p}}$ are equal to $${{N}_{{{{d [p]}}}}}=p+1= {{\overline{N}}_{{{{d [p]}}}}},$$

   when $(\frac{d}{p}) = -1$, and
   $${{N}_{d[p]}} = p-3 = {{\overline{N}}_{{{{d [p]}}}}}-4,$$ when $(\frac{d}{p}) = 1$
iff
 $p \equiv 3 (\bmod \text{ } 4)$ is prime and $\sum\limits_{j=0}^{\frac{p-1}{2}}{{{(C_{\frac{p-1}{2}}^{j})}^{2}}{{d}^{j}}}\equiv 0(\bmod \text{ } p)$.

\end{cor}

Since all transformations in proof of Theorem \ref{first result} were equivalent transitions then we obtain the proof of equivalence of conditions.

\begin{thm}
If the coefficient $d=2$ or $d=2^{-1}$ and $ p \equiv 3 (\bmod  4) $ then $\sum\limits_{j=0}^{\frac{p-1}{2}}{{{d}^{j}}{{(C_{\frac{p-1}{d}}^{j})}^{2}}\equiv 0\,(\bmod p)}$ and
$\overline{N}_{d[p]} = p+1$.
\end{thm}

When $ p \equiv 3 (\bmod \text{ } 4) $, we shall show that $\sum\limits_{j=0}^{\frac{p-1}{2}}{{{d}^{j}}{{(C_{\frac{p-1}{d}}^{j})}^{2}}\equiv 0\,(\bmod \text{ } p)}.$ We multiply each binomial coefficient in this sum by $(\frac{p-1}{2})!$ to obtain after some algebraic manipulation
$
(\frac{p-1}{2})!C_{\frac{p-1}{2}}^{j}=\frac{(\frac{p-1}{2})(\frac{p-1}{2}-1) \cdots (\frac{p-1}{2}-j+1)(\frac{p-1}{2})!}{1\cdot 2 \cdots j} = \\ = (\frac{p-1}{2})(\frac{p-1}{2}-1) \cdots (\frac{p-1}{2}-j+1)[(\frac{p-1}{2})(\frac{p-1}{2}-1) \cdots (j+1)].
$

After applying the congruence ${{(\frac{p-1}{2}-k)}^{2}} \equiv {{(\frac{p-1}{2}+1+k)}^{2}}(\bmod \text{ } p)$ with $0\le k\le \frac{p-1}{2}$ to the multipliers in previous parentheses, we obtain
$[(\frac{p-1}{2})(\frac{p-1}{2}-1) \cdots (j+1)]$.
It yields
$$\Big(\frac{p-1}{2}\Big)\Big(\frac{p-1}{2}-1\Big) \cdots \Big(\frac{p-1}{2}-j+1\Big)\Big[\Big(\frac{p-1}{2}+1\Big) \cdots \Big(\frac{p-1}{2}+\frac{p-1}{2}-j\Big)\Big]{{(-1)}^{\frac{p-1}{2}-j}}.$$
Thus, as a result of squaring, we have:
\begin{equation} \label{square}
{{\Big( \Big(\frac{p-1}{2}\Big)! \text{ } C_{\frac{p-1}{2}}^{j} \Big)}^{2}}\equiv {{\Big(\frac{p-1}{2}-j+1\Big)}^{2}}{{\Big(\frac{p-1}{2}-j+2\Big)}^{2}} \cdots {\Big{(p-j-1\Big)}^{2}}(\bmod \text{ } p).
\end{equation}




It remains to prove that
$$\sum\limits_{j=0}^{\frac{p-1}{2}}{{{(C_{\frac{p-1}{2}}^{j})}^{2}}{{2}^{j}}}\equiv 0(\bmod \text{ } p)$$ if $p \equiv 3(\bmod \text{ } 4)$.
Consider the auxillary polynomial $P(t) = {(\frac{p-1}{2}!)}^{2} \sum\nolimits_{j=0}^{\frac{p-1}{2}}{{{(C_{\frac{p-1}{2}}^{j})}^{2}}{{t}^{j}}} $. We are going to show that $P(2) = 0$ and therefore $a_{p-1} \equiv 0 (\bmod \text{ } p)$. Using (\ref{square}) it can be shown that
  $a_{p-1}= P(t) = {{(\frac{p-1}{2}!)}^{2}} \sum\nolimits_{j=0}^{\frac{p-1}{2}} {{{(C_{\frac{p-1}{2}}^{j})}^{2}}{{t}^{j}}} \equiv  \sum\nolimits_{j=0}^{\frac{p-1}{2}} ({k+1})^2 ({k+2})^2 ... ({\frac{p-1}{2}} + k)^2 t^k (\bmod \text{ } p)$ over $F_p$.
\\
\\
We replace $d$ by $t$ in (\ref{Super}) such that we can research a more generalised problem. It should be noted that
$P(t)={{\partial}^{\frac{p-1}{2}}} \Big( {{\partial}^{\frac{p-1}{2}}} \big( Q(t) \text{ } t^{\frac{p-1}{2}} \big) t^{\frac{p-1}{2}}\Big)$ over $F_p$, where $Q(t)={{t}^{p-1}}+...+t+1$ and ${{\partial}^{\frac{p-1}{2}}}$ denotes the $\frac{p-1}{2}$-th derivative by $t$, where $t$ is new variable but not a coordinate of curve. Observe that $Q(t)=\frac{{{t}^{p}}-1}{t-1} \equiv \frac{{{(t-1)}^{p}}}{t-1} \equiv {{(t-1)}^{p-1}}(\bmod \text{ } p)$ and therefore the equality $P(t)={{\Big( \big( {{(t-1)}^{p-1}}{{t}^{\frac{p-1}{2}}}\big)^{(\frac{p-1}{2})} {{t}^{\frac{p-1}{2}}}\Big)}^{(\frac{p-1}{2})}}$ holds over $F_p$.
\\
\\
In order to simplify notation we let $\theta=t - 1$ and $R(\theta)=P(\theta+1)$. For the case $t=2$ we have $\theta=1$. Performing this substitution leads the polynomial $P(t)$ of 2 to the polynomial $R(t-1)$ of 1.
  Taking into account the linear nature of the substitution $\theta=t-1$, it can be seen that that derivation by $\theta$ and $t$ coincide.
Derivation leads us to the transformation of polynomial $R(\theta)$ to form where it has the necessary coefficient ${{a}_{p-1}}$.
Then

$$R(\theta)=P(\theta +1)=\partial^{\frac{p-1}{2}} {\Big({{\partial^{\frac{p-1}{2}} \big({{\theta}^{p-1}} {(\theta+1)}^{\frac{p-1}{2}} \big)}} {{(\theta +1)}^{\frac{p-1}{2}}}\Big)} = \partial^{\frac{p-1}{2}} \Big( {\frac{(p-1)!}{((p-1)/2)!}}\theta^{\frac{p-1}{2}} (\theta + 1)^{\frac{p-1}{2}}\Big).$$ 

In order to prove that $a_{p-1}\equiv 0 (\bmod \text{ } p)$, it is now sufficient to s  how that $R(\theta )=0$ if $\theta =1$ over $F_p$. We obtain

\begin{equation} \label{sum}
\begin{array}{l}
R(1)=\frac{(p-1)!}{(\frac{p-1}{2})!}\sum\nolimits_{j=0}^{\frac{p-1}{2}}{C_{\frac{p-1}{2}}^{j}(j+1) \cdots (j+\frac{p-1}{2})}.
\end{array}
\end{equation}

We now will manipulate with the expression $(\frac{p-1}{2}-j+1) (\frac{p-1}{2}-j+2) \cdots (\frac{p-1}{2}-j+\frac{p-1}{2})$. In order to illustrate the simplification we now consider the scenario when $p=11$ and hence $\frac{p-1}{2} = 5$. The expression gets the form $(5 - j + 1)(5-j+2) \cdots (5 - j +5) = (6 - j)(7 - j) \cdots (10 - j) \equiv \big( (-5 - j)(-4 - j) \cdots (-1 - j) \big) \equiv (-1)^5 \big( (j+1)(j+2) \cdots (j+5) \big)
(\bmod \text{ } 11).$

Therefore, for a prime $p$, we can rewrite the expression as $(\frac{p-1}{2}-j+1) (\frac{p-1}{2}-j+2) \cdots (\frac{p-1}{2}-j+\frac{p-1}{2}) \equiv {{(-1)}^{\frac{p-1}{2}}}(j+1) \cdots (j+\frac{p-1}{2}) \equiv -1(j+1) \cdots (j+\frac{p-1}{2}) (\bmod \text{ } p)$.


As a result, the symmetrical terms in (\ref{sum}) can be reduced yielding $a_{p-1} \equiv 0 (\bmod \text{ } p)$. It should be noted that ${{(-1)}^{\frac{p-1}{2}}}=-1$ since $p=Mk+3$ and $\frac{p-1}{2}=2k+1$.
Consequently, we have $P(2)=R(1)=0$ and hence $a_{p-1} \equiv 0 (\bmod \text{ } p)$ as required. Thus, $\sum\nolimits_{j=0}^{\frac{p-1}{2}}{{{(C_{\frac{p-1}{2}}^{j})}^{2}}\equiv 0}(\bmod \text{ } p) $, completing the proof of the of the theorem.
\end{proof}

\begin{cor} \label{supersingpair}
The curve $E_d$ is supersingular iff $E_{d^{-1}}$ is supersingular.
\end{cor}
\begin{proof} Let us recall
 the proved fact in Theorem \ref{first result} that ${{N}_{{{{d [p]}}}}} \equiv  -{{a}_{2p-2}}-{{a}_{p-1}}\equiv -\left(\frac{d}{p}\right)+\sum\limits_{j=0}^{\frac{p-1}{2}}{{{(C_{\frac{p-1}{2}}^{j})}^{2}}{{d}^{j}}}(\bmod \text{ } p)$.
 Since ${{{(C_{\frac{p-1}{2}}^{j})}^{2}}{{d}^{j}}} \equiv 0 (\bmod \text{ } p)$ by condition, and the congruence $(\frac{d}{p})\equiv (\frac{d^{-1}}{p})$ holds, then ${{N}_{{{{d [p]}}}}} \equiv {{N}_{{{{d^{-1} [p]}}}}}$.
\end{proof}

Now  we estimate the number of points on the curve (\ref{Prod}). Let ${{M}_{{{{d [p]}}}}}$ denote the number of solutions to equation (\ref{Prod}) over the field ${{\text{F}}_{p}}$. It should be observed that for $x=1$ and $x=-1$, the right side of (\ref{Prod}) is equal to 0.
Due to this the number ${{M}_{{{{d [p]}}}}}$ can therefore be bounded by
\begin{equation} \label{bounds}
\begin{array}{l}
2\le {{M}_{{{{d [p]}}}}}\le 2p-2,
\end{array}
\end{equation}

where if ${{a}_{p-1}}\equiv 0(\bmod \text{ } p)$ we have ${{N}_{{{{d [p]}}}}}\equiv -  \big(\frac{d}{p} \big)(\bmod \text{ }p)$. The number of solutions is bounded by ${{N}_{{{{d [p]}}}}}\le 2p-2$ because if $x=1$ and $x=-1$ we only have one value of $y$, namely $y=0$. For different values of $x$, we will have no more than two solutions for $y$ because the equation (\ref{Prod}) is quadratic relative to $y$. Thus, the only possible number is ${{M}_{{{{d [p]}}}}}\equiv p-\big( \frac{d}{p} \big) (\bmod \text{ } p)$.



\begin{cor}
If $p \equiv 3 (\bmod \text{ } 4)$, is prime then there exists some $T$ such that $ T \equiv \sum\limits_{j=0}^{\frac{p-1}{2}}{{{(C_{\frac{p-1}{2}}^{j})}^{2}}{{d}^{j}}} \le 2 \sqrt{q}$
and ${{N}_{{{{d [p]}}}}} =  p-1 - 2\left(\frac{d}{p}\right) + T$.
\end{cor}

\begin{proof}
Due to equality \eqref{general numb} and the bounds \eqref{bounds} as well as according to generalized Hasse-Weil theorem $|N_{d[p]} - (p+1)- 2\left(\frac{d}{p}\right)| \le 2g \sqrt{p}$, where $g$ is genus of curve, we obtain exact number ${{N}_{{{{d [p]}}}}}$. As we showed, $g=1$.
 From Theorem \ref{first result} as well as from Corollary \ref{supersingpair} we get, that $ \sum\limits_{j=0}^{\frac{p-1}{2}}{{{(C_{\frac{p-1}{2}}^{j})}^{2}}{{d}^{j}}} \equiv - N_{d[p]} - (p+1)- 2\left(\frac{d}{p}\right)$ so there exists $T\in \mathbb{Z}$, such that $T < 2 \sqrt{p} $ and $N_{d[p]} = p-1- 2\left(\frac{d}{p}\right) +T$.
\end{proof}

\begin{exam}
If $p=13$, $d=2$ gives $N_{2[13]}=8$ and $p=13$, $d^{-1}=7$ gives that the number of points of $E_7$ is $N_{7[13]}=20$, which is in contradiction to that suggested by Bessalov and Thsigankova 
 \cite{B}. Moreover, if $p\equiv 7(\bmod \text{ } 8)$, then the order of torsion subgroup of curve is $N_{2}=N_{2^{-1}}=p-3$, which is clearly different to $p+1$ as suggested by \cite{B}.

For instance $p=31$, then $N_{2[31]}=N_{2^{-1}[31]}= 28 = 31 - 3$, which is clearly not equal to $p+1$. If $p=7, d=2^{-1}\equiv (4 \bmod \text{ } 7)$ then the curve $E_{2^{-1}}$ has four points, namely $(0,1); (0,6);
(1,0); (6,0)$, and the in case $p=7$ with $d=2 (\bmod \text{ } 7)$, the curve $E_{2^{-1}}$ also has four points: $(0,1); (0,6); (1,0); (6,0)$, demonstrating the order in this scenario is $p-3$.
\end{exam}
\vspace{5.0mm}

The following theorem shows that the total number of affine points upon the Edwards curves $E_d$ and $E_{d^{-1}}$ are equal under certain assumptions. This theorem additionally provides us with a formula for enumerating the number of affine points upon the birationally isomorphic Montgomery curve $N_M$.

\begin{thm}\label{extnumb} Let $d$ satisfy the condition of supersingularity (\ref{Super}).
If $n \equiv 1(\bmod \text{ } 2)$ and $p$ is prime, then
$$\overline{N}_{d[p^n]} = p^n + 1,$$
and the order of curve is equal to

$$N_{d[p^n]} = p^n - 1 - 2 \left( \frac{d}{p} \right).$$
 If $n \equiv 0(\bmod \text{ } 2)$ and $p$ is prime, then the order of curve 
  $$N_{d[p^n]} = p^{n}-3-2(-p)^{\frac{n}{2}},$$
and the order of projective curve is equal to $$\overline{N}_{d[p^n]} = p^{n}+1-2(-p)^{\frac{n}{2}}.$$

\end{thm}


\begin{proof}

We consider the extension of the base field $F_p$ to $F_{p^n}$ in order to determine the number of the points on the curve $x^2 + y^2 = 1 + dx^2y^2$. Let $P (x)$ denotes a polynomial with degree $m> 2$ whose coefficients are from ${{\mathbb{F}}_{p}}$. To make the proof, we take into account that it is known \cite{Step} that the number of solutions to $y^2 = P (x)$ over ${{\mathbb{F}}_{{{p}^{n}}}}$ will have the form ${{p}^{n}}+ 1 -\omega _{1}^{n}-...-\omega _{m-1}^{n}$ where ${{\omega }_{1}},...,{{\omega }_{m-1}}\in \mathbb{C}$, $|\omega_i| = p^{\frac {1}{2}}$.

In case of our supersingular curve, if $ n \equiv 1(\bmod \text{ } 2)$ the number of points on projective curve over ${{\mathbb{F}}_{{{p}^{n}}}}$ is determined by the expression ${{p}^{n}}+ 1 -\omega_{1}^{n}-\omega _{2}^{n}$, where $\omega_i^n \in \mathbb{C}$ and $\omega_1 = - \omega_2$, $|\omega_i| = \sqrt{p}$  that's why $\omega_1 = i \sqrt{p}$,  $\omega_2= - i \sqrt{p}$  with $i \in \{1,2\}$. In the general case, it is known \cite{Step, Gla, Lidl} that $|\omega_i| = p^{\frac {1}{2}}$. The order of the projective curve is therefore ${{p}^{n}}+ 1$.

If $p \equiv 7 (\bmod \text{ } 8)$, then it is known from a result of Skuratovskii \cite{SkRMM} that $E_d$ has in its projective closure of the curve singular points which are not affine and therefore
$${{N}_{{{{d [p]}}}}}={{p}^{n}}-3.$$
If $p \equiv 3(\bmod \text{ } 8)$, then there are no singular points, hence $${\overline{N}}_{{{{d [p]}}}} = {{N}_{{{{d [p]}}}}}={{p}^{n}}+1.$$

Consequently the number of points on the Edwards curve depends on $\big(\frac{d}{p}\big)$ and is equal to $N_{d[p]}=p^n-3$ if $p\equiv 7(\bmod \text{ } 8)$ and $N_{d[p]}=p^n+1$ if $p\equiv 3(\bmod \text{ } 8)$ where $n \equiv 1(\bmod \text{ } 2)$. We note that this is because the transformation of (\ref{Prod}) in $E_d$ depends upon the denominator $(dx^2-1)$.

If $n \equiv 1(\bmod \text{ } 2)$ then, with respect to the sum of root of  of the characteristic equation for the Frobenius endomorphism $\omega_{1}^{n}+\omega _{2}^{n}$, which in this case have the same signs, we obtain that the number of points in the group of points of the curve is ${{p}^{n}}+ 1 -\omega_{1}^{n} - \omega_{2}^{n}$ \cite{Del}.


For $ n \equiv 0(\bmod \text{ } 2)$ we always have, that every $d\in F_p$ is a quadratic residue in $F_{p^n}$. Consequently, because of $(\frac{d}{p})=1$ four singular points appear on the curve. Thus, the number of affine points is less by 4, i.e. $N_{d[p^n]} = p^n - 1 - 2 \left( \frac{d}{p} \right) - 2(-p)^\frac{n}{2} = p^n - 3 - 2(-p)^\frac{n}{2}.$
In more details ${{\omega }_{1}},\,\,\,{{\omega }_{2}}$ are eigen values of the Frobenius operator $F$ endomorphism on etale cohomology over the finite field ${{\mathbb{F}}_{{{p}^{n}}}}$, where $F$ acts of ${{H}^{i}}(X)$. The number of points, in general case, are determined by Lefshitz formula:
$$\#X\left( {{\text{F}}_{{{p}^{n}}}} \right)=\sum\limits_{{}}{{{(-1)}^{i}}tr({{F}^{n}}\left| {{H}^{i}}(X) \right.)}$$
where $\#X\left( {{\text{F}}_{{{p}^{n}}}} \right)$ is a number of points in the manifold $X$ over ${{\mathbb{F}}_{{{p}^{n}}}}$,  ${{F}^{n}}$ is composition of Frobenius operator. In our case, ${{E}_{d}}$ is considered as the manifold $X$ over ${{\mathbb{F}}_{{{p}^{n}}}}$.

\begin{lem}
 There exists birational isomorphism between ${{E}_{d}}$ and ${{E}_{M}}$, which is determined by correspondent mappings $x=\frac{1+u}{1-u}$ and $y=\frac{2u}{v}$.
\end{lem}

\begin{proof} To verify this statement in supersingular case
 we suppose that the curve \[{{x}^{2}}+{{y}^{2}}=1+d{{x}^{2}}{{y}^{2}}\] contains $p-1-2\left( \frac{d}{p} \right)$ points $(x, y)$, with coordinates over prime field ${{F}_{p}}$.  Consider the transformation of the curve ${{x}^{2}}+{{y}^{2}}=1+d{{x}^{2}}{{y}^{2}}$ into the following form ${{y}^{2}}(d{{x}^{2}}-1)={{x}^{2}}-1$.  Make the substitutions $x=\frac{1+u}{1-u}$ and $y=\frac{2u}{v}.$
 We will call the special points of this transformations the point in which these transformations or inverse transformations are not determined.
 As a result the equation of curve the equation of the curve takes the form \[\frac{4{{u}^{2}}}{{{v}^{2}}}\cdot \frac{(d-1){{u}^{2}}+2(d+1)u+(d-1)}{{{(1-u)}^{2}}}=\frac{4u}{{{(1-u)}^{2}}}.\]
 Multiply the equation of the curve by \[\frac{{{v}^{2}}{{(1-u)}^{2}}}{4u}.\]
  As a result of the reduction, we obtain the equation \[{{v}^{2}}=(d-1){{u}^{3}}+2(d+1){{u}^{2}}+(d-1)u.\] We analyze what new solutions appeared in the resulting equation in comparing with ${{y}^{2}}(d{{x}^{2}}-1)={{x}^{2}}-1$.

First, there is an additional solution $(u, v) = (0, 0)$.
Second, if $ d $ is a quadratic residue by modulo $ p $, then the solutions appear

\[({{u}_{1}},{{v}_{1}})=\left( \frac{-(d+1)-2\sqrt{d}}{d-1},\text{ }0 \right),\,\,\,\,({{u}_{2}},{{v}_{2}})=\left( \frac{-(d+1)+2\sqrt{d}}{d-1},\text{ }0 \right).\]

If $\left( \frac{d}{p} \right)=-1$ then as it was shown above the order of $E_d$ is equal to $p+1$.
Therefore, in case $\left( \frac{d}{p} \right)=-1$ order of $E_d$ appears one additional solution of from $(u, 0)$ more exact it is point with coordinates $\left( 0,\,\,0 \right)$ also two points $((-1;0), (1;0))$ of $E_d$ have not images on $E_M$ in result of action of birational map on $E_M$. Thus, in this case, number of affine points on $E_M$ is equal to $p+1-2+1=p$.
The table of correspondence between points is the following.


\begin{table}[h!]

	\begin{center}
	\begin{tabular}{ |c|c| }
	\hline
	Special points of $E_M$ & Special points of $E_d$ \\
	\hline
	\hline
	$(0,0)$ & - \\
	\hline
	$(\frac {-\left( d+1\right) -2\sqrt {d}}{d-1},0)$ & - \\
	\hline
	$(\frac {-\left( d+1\right) + 2\sqrt {d}}{d-1},0)$ & - \\
	\hline
	$(1, -2\sqrt{d})$ & - \\
	\hline
	$(1, 2\sqrt{d})$ & - \\
	\hline
	- & $(-1, 0)$ \\
	\hline
	- & $(1, 0)$ \\
	\hline
	\end{tabular}
	\caption{Special points of birational maping}
	\end{center}
	\end{table}


If $x= -1$ then equality $x = \frac{1+u}{1-u}$ transforms to form $-1+u=1+u$, or $-1=1$ that is impossible for $p>2$.
therefore point $(-1, 0)$ have not an image on $E_M$.

Consider the case $x=1$. As a result of the substitutions $x = (1 + u)/(1-u),   y = 2u /v$
   we get the pair $(x, y)$ corresponding to the pair
$(u, v)$ for which $v^2 = (d-1) u ^ 3 + 2 (d + 1) u ^ 2 + (d-1) u$.

      If it occurs  that $y=0$, then the preimage having coordinates $u=0$ and $v$  is not equal to 0 is suitable for the birational map $y=\frac{2u}{v}$ which implies that  $u=0$  and $v\ne 0$.
But  pair $(u, v)$ of such form do not satisfies the equation of obtained in result of mapping equation of Montgomery curve  ${{v}^{2}}=\left( d-1 \right){{u}^{3}}+2\left( d+1 \right){{u}^{2}}+(d-1)u$.
Therefore the corresponding point $(u, v)$ will not be a solution to the equation
$v ^ 2 = (d-1) u ^ 3 + 2 (d + 1) u ^ 2 + (d-1) u$, since there will be an element on the left side,
different from 0, and on the right will be 0. That is a contradiction as required, therefore $(x,y)=(1, 0)$ is the special point having not image on $E_M$.

If $y= 0$ then in equality $y = \frac{2u}{v}$ appear zeros in  numerator and denominator and transformation is not correct. 

The points $ (\frac {-\left( d+1\right) -2\sqrt {d}}{d-1},0)$, $(\frac {-\left( d+1\right) + 2\sqrt {d}}{d-1},0)$,
$(1, -2\sqrt{d})$, $(1, 2\sqrt{d})$ exist on $E_M$ only when $(\frac{d}{p})=1$. These points are elements of group which can be presented on Riemann sphere over $F_q$. The points $(1, -2\sqrt{d})$, $(1, 2\sqrt{d})$ have not images on $E_d$ because of in denominator of transformations $x = \frac{1+u}{1-u}$ appears zero. By the same reason points $(\frac {-\left( d+1\right) -2\sqrt {d}}{d-1},0)$, $(\frac {-\left( d+1\right) + 2\sqrt {d}}{d-1},0)$ have not an images on $E_d$.




If $\left( \frac{d}{p} \right)=1$ then as it was shown above the order of $E_d$ is equal to $p-3$. Therefore
order of $E_M$ is equal to $p$ because of 5  additional solutions of equation of $E_M$ appears but 2 points $((-1;0), (1;0))$ of $E_d$ have not images on $E_M$.
These are 5 additional points appointed in tableau above. Also it exists one infinitely distant point on an Montgomery curve. Thus, the order of $E_M$ is equal $p+1$ in this case as supersingular curve has.
\end{proof}

It should be noted that the supersingular curve $E_d$ is birationally equivalent to the supersingular elliptic curve which may be presented in Montgomery form ${{v}^{2}}=(d-1){{u}^{3}}+2(d+1){{u}^{2}}+(d-1)u$. As well as exceptional points \cite{Bir} for the birational equivalence $(u, v)\mapsto (2u/v, (u-1)/(u+1))= (x, y)$ are in one to one correspondence to the affine point of order 2 on $E_{d}$ and to the points in projective closure of $E_d$. Since the formula for number of affine  points on $E_M$ can be applied to counting $N_{d[p]}$. In such way we apply this result \cite{Step, Mon87}, to the case $y^2 = P (x)$, where $degP(x)=m$, $m=3$. 
The order $N_{M[p^n]}$ of the curve $E_M$ over $F_{p^k}$ can be evaluated due to Stepanov \cite{Step, Gla}. The research tells us that the order is $\overline{N}_{M[p^n]} = p^n + 1- \omega_1^n- \omega_2^n$, where $\omega_i^n \in \mathbb{C}$ and $\omega_1^n = - \omega_2^n$, $|\omega_i| = \sqrt{p}$ with $i \in \{1,2\}$. Therefore, we conclude when $n \equiv 1 (\bmod \text{ } 2)$, we know the order of Montgomery curve is precisely $N_{M[p^n]} = p^n+1$.
  This result leads us to the conclusion that the number of solutions of ${{x}^{2}}+{{y}^{2}}=1+d{{x}^{2}}{{y}^{2}}$ as well as ${{v}^{2}}=(d-1){{u}^{3}}+2(d+1){{u}^{2}}+(d-1)u$ over the finite field ${{\mathbb{F}}_{{{p}^{n}}}}$ are determined by the expression ${{p}^{n}} + 1 - \omega_{1}^{n} - \omega _{2}^{n}$ if $ n \equiv 1(\bmod \text{ } 2)$.
\end{proof}

\begin{exam}
The elliptic curve presented in the form of Montgomery $E_M : {{v}^{2}}={{u}^{3}}+6{{u}^{2}}+u$, is birationally equivalent \cite{Bir} to the curve ${{x}^{2}}+{{y}^{2}}=1+{{2}}{{x}^{2}}{{y}^{2}}$ over the field $F_{p^k}$.
\end{exam}

\begin{cor}\label{corext}
If $d=2$, $n \equiv 1(\bmod \text{ } 2)$ and $p \equiv 3(\bmod \text{ } 8)$, then the order of curve $E_d$ and order of the projective curve are the following: $$N_{d[p^n]} = p^n + 1,  \, \overline{N}_{d[p^n]} = p^n + 1.$$


If $d=2$, $n \equiv 1(\bmod \text{ } 2)$ and $p \equiv 7(\bmod \text{ } 8)$, then the number of points of projective curve is $$\overline{N}_{d[p^n]} = p^n + 1,$$
and the number of points on affine curve $E_d$ is also $$N_{d[p^n]} = p^n - 3.$$

In case $d=2$, $n \equiv 0(\bmod \text{ } 2)$, $p \equiv 3(\bmod \text{ } 4)$, the general formula of the curves order is $${N}_{d[p^n]} = p^{n} - 3  -2(-p)^{\frac{n}{2}}.$$


The general formula for $n \equiv 0(\bmod \text{ } 2)$ and $d=2$ for the number of points on projective curve for the supersingular case is $$\overline{N}_{d[p^n]} = p^n + 1-2 (-p)^{\frac{n}{2}}.$$
\end{cor}


\begin{proof}
We denote by $N_{M[p^n]}$ the order of the curve $E_M$ over $F_{p^n}$. The order $N_{M[p^n]}$ of $E_M$ over $F_{p^n}$ can be evaluated \cite{Lidl} as $N_{M[p^n]} = p^n +1 - \omega_1^n - \omega_2^n$, where $\omega_i^n \in \mathbb{C}$ and $\omega_1^n = - \omega_2^n$, $|\omega_i|= \sqrt{p}$ with $i\in \{1,2\}$. For the finite algebraic extension of degree $n$, we will consider ${{p}^{n}}-\omega _{1}^{n}-\omega _{2}^{n}={{p}^{n}}$ if $ n \equiv 1(\bmod \text{ } 2)$. Therefore, for $n\equiv 1 (\bmod 2)$, the order of the Montgomery curve is precisely given by $N_{M[p^n]} = p^n +1$. Here's one infinitely remote point as a neutral element of the group of points of the curve.

Considering now an elliptic curve, we have ${{\omega }_{1}}={{\bar{\omega }}_{2}}$ by \cite{Kob}, which leads to ${{\omega }_{1}}+{{\omega }_{2}}=0$. For $n=1$, it is clear that ${{N}_{M}}=p$. When $n$ is odd, we have $\omega_1^n + \omega_2^n=0$ and therefore ${{N}_{M,n}}={{p}^{n}}+1$. Because $n$ is even by initial assumption, we shall show that ${{N}_{M[p^n]}}={{p}^{n}}+1-2{{(-p)}^{\frac{n}{2}}}$ holds as required.
\end{proof}

Note that for $n=2$ we can express the number as ${{\overline{N}}_{{{d[p^2]}}}}={{p}^{2}}+1+2p={{\left( p+1 \right)}^{2}}$ with respect to Lagrange theorem have to be divisible on ${{\bar{N}}_{{d[p]}}}$.  Because a  group of ${{E}_{d}}({{F}_{{{p}^{2}}}})$ over square extension of ${{F}_{p}}$  contains the group ${{E}_{d}}({{\text{F}}_{p}})$ as a proper subgroup. In fact, according to Theorem 1 the order ${{E}_{d}}({{\text{F}}_{p}})$ is $p+1$ therefore divisibility of order ${{E}_{d}}({{\text{F}}_{{{p}^{2}}}})$ holds because in our case $p=7$ thus ${{\overline{N}}_{{{E}_{d}}}}={8}^{2}$ and  $p+1=8=N_{d[7]}$ \cite{Var}.

The following two examples exemplify Corollary \ref{corext}.

\begin{exam}
If $p \equiv 3(\bmod \text{ } 8)$ and $n=2k$ then we have when $d=2$, $n=2$, $p=3$ that the number of affine points equals to $${{N}_{2[3]}}={{p}^{n}} -3 -2{{(-p)}^{\frac{n}{2}}} = 3^2 -3 -2 \cdot(-3) = 12,$$
and the number of projective points is equal to $${{\overline{N}}_{2[3]}}={{p}^{n}} + 1 -2{{(-p)}^{\frac{n}{2}}} = 3^2 + 1 - 2 \cdot (-3) = 16.$$
\end{exam}

\begin{exam}
If $p \equiv 7(\bmod \text{ } 8)$ and $n=2k$ then we have when $d=2$, $n=2$, $p=7$ that the number of affine points equals to $${{N}_{2[7]}}={{p}^{n}} -3 -2{{(-p)}^{\frac{n}{2}}} = 7^2 - 3 -2 \cdot(-7) = 60,$$
and the number of projective points is equal to $${{\overline{N}}_{2[7]}}={{p}^{n}} + 1 -2{{(-p)}^{\frac{n}{2}}} = 7^2 + 1 - 2 \cdot (-7) = 64.$$
\end{exam}

\begin{prop}
The group of points of the supersingular curve $E_d$ contains $p-1 - 2\left(\frac{d}{p}\right)$ affine points and the affine singular points whose number is $2\left(\frac{d}{p}\right)+2$.  
\end{prop}

\begin{proof}
The singular points were discovered in \cite{SkRMM} and hence if the curve is free of singular points then the group order is $p+1$.
\end{proof}

\begin{exam}
The number of curve points over finite field when $d=2$ and $p=31$ is equal to ${{N}_{{2[31]}}}={{N}_{{2}^{-1}[31]}}=p-3=28$.
\end{exam}


\begin{thm}
The order of Edwards curve over $F_p$ is congruent to
$${{\overline{N}}_{{{{d [p]}}}}} \equiv \Big(p-1 - 2\left(\frac{d}{p}\right) +{{(-1)}^{\frac{p+1}{2}}}\sum\limits_{j=0}^{\frac{p-1}{2}}{{{\big(C_{\frac{p-1}{2}}^{j}\big)}^{2}}{{d}^{j}}}\Big) \equiv \Big({{(-1)}^{\frac{p+1}{2}}}\sum\limits_{j=0}^{\frac{p-1}{2}}{{{\big(C_{\frac{p-1}{2}}^{j}\big)}^{2}}{{d}^{j}} -1 - 2\left(\frac{d}{p}\right)}\Big)(\bmod \text{ } p ).$$
The true value of ${{\overline{N}}_{{{{d [p]}}}}}$ lies in $[4; 2p]$ and is even.
\end{thm}

 \begin{proof}

This result follows from the number of solutions of the equation $\begin{array}{*{35}{l}}
   {{y}^{2}}=(d{{x}^{2}}-1)({{x}^{2}}-1)  \\
\end{array}$ over ${{F}_{p}}$  which equals to
\begin{align*}
  & \sum\limits_{x=0}^{p-1}{\left( \frac{({{x}^{2}}-1)(d{{x}^{2}}-1)}{p})+1 \right)}\equiv \sum\limits_{x=0}^{p-1}{(\frac{({{x}^{2}}-1)(d{{x}^{2}}-1)}{p}))+p}\equiv
\\
 &
\equiv
(\sum\limits_{j=0}^{\frac{p-1}{2}}{{{({{x}^{2}}-1)}^{\frac{p-1}{2}}}{{(d{{x}^{2}}-1)}^{\frac{p-1}{2}}})}\left( \bmod p \right) \equiv
\\
 &
\equiv ({{(-1)}^{\frac{p+1}{2}}}\sum\limits_{j=0}^{\frac{p-1}{2}}{{{(C_{\frac{p-1}{2}}^{j})}^{2}}{{d}^{j}} -(\frac{d}{p}) )}\left( \bmod p \right). \\
\end{align*}

 The quantity of solutions for ${{x}^{2}}+{{y}^{2}}=1+d{{x}^{2}}{{y}^{2}}$  differs from the quantity  of $    {{y}^{2}}=(d{{x}^{2}}-1)({{x}^{2}}-1)$  by  $(\frac{d}{p})+1$  due to new solutions in the from $(\sqrt{d},0),\,\,(-\sqrt{d},\,0)$.   So this quantity is such
\begin{align*}
  & \sum\limits_{x=0}^{p-1}{\left( \frac{({{x}^{2}}-1)(d{{x}^{2}}-1)}{p})+1 \right)}-\left( (\frac{d}{p})+1 \right)\equiv
\\
 &
 \sum\limits_{x=0}^{p-1}{(\frac{({{x}^{2}}-1)(d{{x}^{2}}-1)}{p}))+p- \left( (\frac{d}{p}) - 1 \right)}\equiv
\\
 &
  \equiv
 (\sum\limits_{j=0}^{\frac{p-1}{2}}{{{({{x}^{2}}-1)}^{\frac{p-1}{2}}}{{(d{{x}^{2}}-1)}^{\frac{p-1}{2}}}-(\frac{d}{p})+1)}\left( \bmod p \right)
 \equiv
\\
 &
\equiv {{(-1)}^{\frac{p+1}{2}}}\sum\limits_{j=0}^{\frac{p-1}{2}}{{{(C_{\frac{p-1}{2}}^{j})}^{2}}{{d}^{j}}-(2(\frac{d}{p})+1)}\left( \bmod p \right). \\
\end{align*}
 According to Lemma 1 the last sum $(\sum\limits_{j=0}^{\frac{p-1}{2}}{{{({{x}^{2}}-1)}^{\frac{p-1}{2}}}{{(d{{x}^{2}}-1)}^{\frac{p-1}{2}}})}\left( \bmod p \right)$ is congruent to
$-{{a}_{p-1}}-{{a}_{2p-2}}(mod\,\,p)$, where ${{a}_{i}}$ are the coefficients from presentation  \[{{({{x}^{2}}-1)}^{\frac{p-1}{2}}}{{(d{{x}^{2}}-1)}^{\frac{p-1}{2}}}={{a}_{0}}+{{a}_{1}}x+...+{{a}_{2p-2}}{{x}^{2p-2}}.\]
Last presentation was obtained due to transformation ${{({{x}^{2}}-1)}^{\frac{p-1}{2}}}{{(d{{x}^{2}}-1)}^{\frac{p-1}{2}}}= (\sum\limits_{x=0}^{p-1}{C_{\frac{p-1}{2}}^{k}{{x}^{2k}}{{(-1)}^{\frac{p-1}{2}-k}})\,}(\sum\limits_{x=0}^{p-1}{C_{\frac{p-1}{2}}^{j}{{d}^{j}}{{x}^{2j}}{{(-1)}^{\frac{p-1}{2}-j}})}$.
Therefore ${{a}_{2p-2}}$ is equal to ${{d}^{\frac{p-1}{2}}}\equiv (\frac{d}{p})(\bmod p)$ and  ${{a}_{p-1}}=\sum\nolimits_{j=0}^{\frac{p-1}{2}}{{{(C_{\frac{p-1}{2}}^{j})}^{2}}{{d}^{j}}{{(-1)}^{\frac{p-1}{2}}}}$.

According to Newton's binomial formula ${{a}_{p-1}}$ equal to the coefficient at ${{x}^{p-1}}$ in the product of two brackets and when substituting it $d$ instead of 2  is such

\[{{(-1)}^{\frac{p-1}{2}}}\sum\limits_{j=0}^{\frac{p-1}{2}}{{{d}^{j}}{{(C_{\frac{p-1}{2}}^{j})}^{2}}},\] that is, it has the form of the polynomial with inverse order of coefficients.

Indeed, we have equality
\begin{align*}
  & \sum\limits_{j=0}^{\frac{p-1}{2}}{{{d}^{j}}(C_{\frac{p-1}{2}}^{\frac{p-1}{2}-j}){{(-1)}^{\frac{p-1}{2}-(\frac{p-1}{2}-j)}}\cdot {{(C_{\frac{p-1}{2}}^{j})}^{2}}{{(-1)}^{\frac{p-1}{2}-j}}}=
  \\
 &
  ={{(-1)}^{\frac{p-1}{2}}}\sum\limits_{j=0}^{\frac{p-1}{2}}{{{d}^{j}}C_{\frac{p-1}{2}}^{\frac{p-1}{2}-j}\cdot C_{\frac{p-1}{2}}^{j}=}{{(-1)}^{\frac{p-1}{2}}}\sum\limits_{j=0}^{\frac{p-1}{2}}{{{d}^{j}}{{(C_{\frac{p-1}{2}}^{j})}^{2}}}. \\
 &  \\
\end{align*}
In form of a sum it is the following
$   \sum\limits_{j=0}^{\frac{p-1}{2}}{{{2}^{j}}(C_{\frac{p-1}{2}}^{\frac{p-1}{2}-j}){{(-1)}^{\frac{p-1}{2}-(\frac{p-1}{2}-j)}}\cdot {{2}^{j}}{{(C_{\frac{p-1}{2}}^{j})}^{2}}{{(-1)}^{\frac{p-1}{2}-j}}}=$ \\
 $={{(-1)}^{\frac{p-1}{2}}}\sum\limits_{j=0}^{\frac{p-1}{2}}{{{2}^{j}}C_{\frac{p-1}{2}}^{\frac{p-1}{2}-j}\cdot C_{\frac{p-1}{2}}^{j}=}{{(-1)}^{\frac{p-1}{2}}}\sum\limits_{j=0}^{\frac{p-1}{2}}{{{2}^{j}}{{(C_{\frac{p-1}{2}}^{j})}^{2}}}.
$

If \[{{a}_{p-1}}=\sum\nolimits_{j=0}^{\frac{p-1}{2}}{{{(C_{\frac{p-1}{2}}^{j})}^{2}}{{d}^{j}}{{(-1)}^{\frac{p-1}{2}}}}\equiv 0(\bmod p),\]
 then as it is was shown by the author in \cite{SkRMM} this curve is supersingular and the number of solutions of the $\begin{array}{*{35}{l}}
   {{y}^{2}}=(d{{x}^{2}}-1)({{x}^{2}}-1)  \\
\end{array}$ over ${{F}_{p}}$  equals to $p-1-2\left( \frac{d}{p} \right)+\left( 1+(\frac{d}{p}) \right)=p-\left( \frac{d}{p} \right)$ and differs from the quantity  of solutions of  $x^2+{y}^{2}= 1 + d{{x}^{2}}{{y}^{2}}$  by  $(\frac{d}{p})+1$  due to new solutions of ${{y}^{2}}=(d{{x}^{2}}-1)({{x}^{2}}-1)$. Thus, in general case if ${{a}_{p-1}}=\sum\nolimits_{j=0}^{\frac{p-1}{2}}{{{(C_{\frac{p-1}{2}}^{j})}^{2}}{{d}^{j}}{{(-1)}^{\frac{p-1}{2}}}}\ne 0$ we have

${{N}_{{{E}_{d}}}}=(p-(\frac{d}{p})-((\frac{d}{p})+1) -{{(-1)}^{\frac{p-1}{2}}}\sum\limits_{j=0}^{\frac{p-1}{2}}{{{(C_{\frac{p-1}{2}}^{\frac{p-1}{2}-j}C_{\frac{p-1}{2}}^{j})}^{2}}{{d}^{j}})}\equiv
 (p-1-{{(-1)}^{\frac{p-1}{2}}}\sum\limits_{j=0}^{\frac{p-1}{2}}{{{(C_{\frac{p-1}{2}}^{j})}^{2}}{{d}^{j}} -2(\frac{d}{p}))}\equiv
({{(-1)}^{\frac{p+1}{2}}}\sum\limits_{j=0}^{\frac{p-1}{2}}{{{(C_{\frac{p-1}{2}}^{j})}^{2}}{{d}^{j}}  -1 - 2(\frac{d}{p}) )}\left( \bmod p \right).$

The exact order is not less than 4 because cofactor of this curve is 4. To determine the order is uniquely enough to take into account that $p$ and $2p$ have different parity. Taking into account that the order is even we chose a term $p$ or $2p$, for the sum which define the order.
 \end{proof}

\begin{exam}
Number of curve points for $d=2$ and $p=31$ equals ${N}_{{{2}[p]}}={N}_{{2}^{-1}[p]}=p-3=28$.
\end{exam}

 \begin{thm}
If  $\left( \frac{d}{p} \right)=1$, then the orders of the curves ${{E}_{d}}$ and ${{E}_{{{d}^{-1}}}}$, satisfies to the following relation
\[~\left| {{E}_{d}} \right|=~\left| {{E}_{{{d}^{-1}}}} \right|.\]
If $\left( \frac{d}{p} \right)=-1$, then ${{E}_{d}}$ and ${{E}_{{{d}^{-1}}}}$ are pair of twisted curves i.e.
  orders of curves ${{E}_{d}}$ and  ${{E}_{{{d}^{-1}}}}$ satisfies to the following relation of duality
\[~\left| {{E}_{d}} \right|+\left| {{E}_{{{d}^{-1}}}} \right|=2p+2.\]

  \end{thm}
\begin{proof}

Let the curve be defined by ${{x}^{2}}+{{y}^{2}}=1+d{{x}^{2}}{{y}^{2}}(modp)$, then we can express ${{y}^{2}}$ in such way:
\begin{align}\label{d}
{{y}^{2}}\equiv ~\frac{{{x}^{2}}-1} {d{{x}^{2}}-1}\left( mod~p. \right)
 \end{align}

For   ${{x}^{2}}+{{y}^{2}}=1+{{d}^{-1}}{{x}^{2}}{{y}^{2}}(modp)$   we could obtain that

\begin{align}\label{d^{-1}}
{{y}^{2}}\equiv ~\frac{{{x}^{2}}-1} {{{d}^{-1}}{{x}^{2}}-1}\left( mod~p \right)
\end{align}

If  $\left( \frac{d}{p} \right)=1$, then for the fixed ${{x}_{0}}$ a quantity  of  $y$ over ${{\text{F}}_{p}}$ can be calculated by the formula $(\frac{\frac{{{x}^{2}}-1}{{{d}^{-1}}{{x}^{2}}-1}}{p})+1$ for $x$ such that ${{d}^{-1}}{{x}^{2}}+1\equiv 0(\bmod p)$.
For solution  $({{x}_{0}},\,{{y}_{0}})$ to ( \ref{Eqw}), we have  the equality  $y_{0}^{2}\equiv ~\frac{x_{0}^{2}-1}{dx_{0}^{2}-1}\left( mod~p \right)$ and we express    $y_{0}^{2}\equiv ~\frac{1-~\frac{1}{x_{0}^{2}}}{1-~\frac{1}{dx_{0}^{2}}}~d{{~}^{-1}}\equiv ~\frac{{{\left( \frac{1}{{{x}_{0}}} \right)}^{2}}-~1}{\frac{1}{d}{{\left( \frac{1}{{{x}_{0}}} \right)}^{2}}-~1}~d{{~}^{-1}}\equiv ~\frac{{{\left( \frac{1}{{{x}_{0}}} \right)}^{2}}-~1}{{{d}^{-1}}{{\left( \frac{1}{{{x}_{0}}} \right)}^{2}}-~1}~{{d}^{-1}}.$			
Observe that
\begin{equation} \label{rel }
{{y}^{2}}=~\frac{{{x}^{2}}-1}{{{d}^{-1}}{{x}^{2}}-1}=\frac{1-{{x}^{2}}}{1-{{d}^{-1}}{{x}^{2}}}=\frac{(\frac{1}{{{x}^{2}}}-1){{x}^{2}}}{((\frac{d}{{{x}^{2}}})-1){{d}^{-1}}{{x}^{2}}}=\frac{(\frac{1}{{{x}^{2}}}-1)}{((\frac{d}{{{x}^{2}}})-1)}d.				 \end{equation}

Thus,  if   $(x_0,  y_0)$ is  solution  of  (\ref{Eqw}),  then $\left( \frac{1}{{{x}_{0}}},~\frac{{{y}_{0}}}{\sqrt{d}} \right)$  is a solution  to (\ref{d^{-1}}) because last transformations determines that $\frac{\text{y}_{0}^{2}}{d}\equiv ~~\frac{{{d}^{-1}}{{\left( \frac{1}{{{x}_{0}}} \right)}^{2}}-~1}{{{\left( \frac{1}{{{x}_{0}}} \right)}^{2}}-~1}modp$. Therefore last transformations $\left( {{x}_{0}},~{{y}_{0}} \right)~\to (\frac{1}{{{x}_{0}}},~\frac{{{y}_{0}}}{\sqrt{d}})=\left( x,~y \right)$ determines isomorphism and bijection.

In case  $\left( \frac{d}{p} \right)=-1$, then every $x\in {{\text{F}}_{p}}$ is such that $d{{x}^{2}}-1\ne 0$  and  ${{d}^{-1}}{{x}^{2}}-1\ne 0$.
If ${{x}_{0}}\ne 0$, then ${{x}_{0}}$ generate 2 solutions of  (\ref{Eqw})  iff  $x_{0}^{-1}$ gives 0 solutions of (\ref{d^{-1}})  because of (\ref {rel }) yields the following relation

\begin{equation} \label{relQ }
(\frac{\frac{{{x}^{2}}-1}{{{d}^{-1}}{{x}^{2}}-1}}{p})=(\frac{\frac{{{x}^{-2}}-1}{d{{x}^{-2}}-1}}{p})(\frac{d}{p})=-(\frac{\frac{{{x}^{-2}}-1}{d{{x}^{-2}}-1}}{p}).
\end{equation}

Analogous reasons give us that ${{x}_{0}}$ give exactly one solution of  (\ref{Eqw})  iff  $x_{0}^{-1}$ gives 1 solutions of  (\ref{d^{-1}}).
Consider the set $x\in \{1,2,....,p-1\}$ we obtain that  the  total amount of solutions of form $(x_{0}^{-1},\,{{y}_{0}})$ that represent point of  (\ref{Eqw}) and pairs of form  $({{x}_{0}},\,{{y}_{0}})$ that represent point of curve (\ref{d^{-1}}) is $2p-2$.
Also we have two solutions of (\ref{Eqw}) of form $(0,1)$ and $(0,-1)$ and two solutions of (\ref{d^{-1}})  that has form $(0,1)$ and $(0,-1)$.
The proof is fully completed.
\end{proof}



\begin{exam} The number of points on $E_d$ for $d=2$ and ${{d}^{-1}} = 2$ with $p=31$ is equal to ${{N}_{{{2[31]}}}}={{N}_{E_{2}^{-1}[31]}} = p-3 = 28$.
\end{exam}

\begin{exam}
The number of points of $E_d$ over $F_p$ for $p=13$ and $d=2$ is given by $N_{2[13]}=8$. In the case when $p=13$ and $d^{-1}=7$ we have that the number of points of $E_7$ is $N_{7[13]}=20$. Therefore, we have that the sum of orders for these curve is equal to $28=2 \cdot 13 + 2$ which confirms our theorem. The set of points over $F_{13}$ when $d=2$ are precisely
$\{(0,1);
(0,12);
(1,0);
(4,4);
(4,9);
(9,4);
(9,9);
(12,0)\}$,
while for $d=7$, we have the set
 $\left\{ (0,1);
(0,12);
(1,0);
(2,4);
(2,9);
(4,2);
(4,11);
(5,6);
(5,7);
 \right.$    \\
$\left.
(6,5);
(6,8);
(7,5);
 (7,8);
(8,6);
(8,7);
(9,2);
(9,11);
(11,4);
(11,9);
(12,0) \right\} $.
\end{exam}

\begin{exam}
If $p=7$ and $d=2^{-1}\equiv 4 (\bmod \text{ } 7)$, then we have $({\frac{d}{p}})=1 $ and the curve $E_{2^{-1}}$ has four points which are $(0,1); (0,6); (1,0); (6,0)$, and the in case $p=7$ for $d=2 (\bmod \text{ }7)$, the curve $E_{2^{-1}}$ also has four points which are $(0,1); (0,6); (1,0); (6,0)$.
\end{exam}

\begin{defn}
We call the embedding degree a minimal power $k$ of a finite field extension such that the group of points of the curve can be embedded in the multiplicative group of ${{\mathbb{F}}_{{{p}^{k}}}}$.
\end{defn}

Let us obtain conditions of embedding \cite{Barret} for the group of supersingular curves ${{E}_{d}}[{{\mathbb{F}}_{p}}]$ of order $p$ in the multiplicative group of field ${{\mathbb{F}}_{{{p}^{k}}}}$ whose embedding degree is $k = 12$ \cite{Barret}. We now utilise the Zsigmondy theorem which implies that a suitable characteristic of field $ {\mathbb{F}}_p$ is an arbitrary prime $p$ which do not divide $12$ and satisfies the condition $q\left| {\text{P}_{12}(p)} \right.$, where ${{\text{P}}_{12}}(x)$ is the cyclotomic polynomial.  This $p$ will satisfy the necessary conditions $({{x}^{n}}-1) \not{|} \text{ } p$ for an arbitrary $n=1,...,11$.

\begin{prop}
The degree of embedding for the group of a supersingular curve $E_d$ is equal to 2.
\end{prop}
The order of the group of a supersingular curve ${{E}_{d}}$ is equal to ${{p}^{k}}+1$. It should be observed that ${{p}^{k}}+1$ divides ${{p}^{2k}} - 1$, but ${{p}^{k}}+1$ does not divide expressions of the form ${{p}^{2l}}-1$ with $l<k$. This division does not work for smaller values of $l$ due to the decomposition of the expression ${{p}^{2k}}-1=({{p}^{k}}-1)({{p}^{k}}+1)$. Therefore, we can use the definition to conclude that the degree of immersion must be 2, confirming the proposition.

Consider ${{\text{E}}_{2}}$ over ${{\text{F}}_{{{p}^{2}}}}$,  for instance we assume  $p=3$.  We define ${{\text{F}}_{9}}$  as  ${{\text{F}}_{3}}(\alpha )$, where $\alpha $ is a root of  ${{x}^{2}}+1=0$  over ${{\text{F}}_{9}}$.    Therefore elements of ${{\text{F}}_{9}}$ have form:  $a+b\alpha $,  where  $a,\,\,b\in {{\text{F}}_{3}}$.  So we assume that $x\in \{\pm (\alpha +1),\,\,\pm (\alpha -1),\,\,\pm \alpha \}$   and check its belonging  to ${{\text{E}}_{2}}$.
 For instance if  $x=\pm (\alpha +1)$ then ${{x}^{2}}={{\alpha }^{2}}+2\alpha +1=2\alpha = -\alpha$.
  Also in this case $y^2=\frac{2\alpha -1}{\alpha -1}= \frac{(2\alpha -1)(\alpha +1)}{(\alpha -1)(\alpha +1)}= \frac{(2\alpha -1)(\alpha +1)}{(\alpha -1)(\alpha +1)}= \frac{\alpha }{ -2} ={\alpha }.$ Therefore the correspondent second coordinate is $y= \pm({\alpha -1})$.
The similar computations lead us to full the following list of curves points.

\begin{table}[ht]
\begin{center}
\begin{tabular}{|c|c|c|c|c|c|c|} 
\hline
$x$ & $\pm 1$ & $0$ & $\pm (\alpha +1)$ & $\pm (\alpha -1)$  \\
\hline
$y$ & $0$ & $\pm 1$ & $\pm (\alpha -1)$ & $\pm (\alpha +1)$  \\
\hline
\end{tabular}
\end{center}
\caption{\label{tab: points} Points of Edwards curve over square extension.}
\end{table}


The total amount is 12 affine points that confirms Corollary \ref{corext} and Theorem \ref{extnumb} because of $p^n-3-2(-p)^{\frac{n}{2}}=3^2-3-2(-3)=12$.




\section{Conclusion}
The new method for order curve counting was founded for Edwards and elliptic curves. The criterion for supersingularity was additionally obtained.










\end{document}